\documentclass[ECP]{ejpecp}

\usepackage[T1]{fontenc}
\usepackage[utf8]{inputenc}

\SHORTTITLE{Sensitivity of the frog model to initial conditions}
\TITLE{Sensitivity of the frog model to initial conditions}

\AUTHORS{Tobias~Johnson\footnote{College of Staten Island, United States of America.
\EMAIL{tobias.johnson@csi.cuny.edu}}
\and
Leonardo~T.~Rolla\footnote{Argentina National Research Council at the University of Buenos Aires, Argentina;
NYU-ECNU Institute of
Mathematical Sciences at NYU Shanghai, China. \EMAIL{leorolla@dm.uba.ar}}}


\AMSSUBJ{60K35; 60J80; 60J10}

\SUBMITTED{September 10, 2018}
\ACCEPTED{April 12, 2019}

\ARXIVID{1809.03082}

\VOLUME{24}
\YEAR{2019}
\PAPERNUM{29}
\DOI{10.1214/19-ECP230}

\ABSTRACT{The frog model is an interacting particle system on a graph. Active
particles perform independent simple random walks, while sleeping particles
remain inert until visited by an active particle. Some number of sleeping
particles are placed at each site sampled independently from a certain
distribution, and then one particle is activated to begin the process. We
show that the recurrence or transience of the model is sensitive not just to
the expectation but to the entire distribution. This is in contrast to
closely related models like branching random walk and activated random walk.}

\DeclarePairedDelimiter\abs{\lvert}{\rvert}\DeclarePairedDelimiter\floor{\lfloor}{\rfloor}\newcommand{\pgfsucc}{\succeq_{\text{pgf}}}
\DeclareMathOperator{\Poi}{Poisson}

\newcommand{\Thom}{\mathbf{T}}
\newcommand{\defeq}{:=}

\begin{document}


\section{Introduction}

The frog model is an interacting particle system on a graph. Initially, one designated
site contains an active particle, and all other sites contain some number of sleeping particles,
typically sampled independently from a given distribution.
Active particles perform simple random walks; these are generally taken to be in discrete time,
though it is irrelevant to this paper.
When an active particle visits a site, all sleeping particles there are activated.
For no deep reason, the particles have come to be called frogs. We represent
a frog model as a pair $(S,\eta )$ where $\eta (v)$ gives the count of sleeping particles on
a vertex $v$, and $S_{t}(v,i)$ gives the path of the $i$th particle on vertex~$v$ for each $1\leq i\leq \eta (v)$.

We call a realization of the frog model \emph{recurrent} if the starting site is visited
infinitely often by particles and \emph{transient} if not. The first question for the frog
model on a given infinite graph is whether it is recurrent or transient.
On $\mathbb {Z}^{d}$, the frog
model is recurrent a.s.\ if the initial configuration $(\eta (v))_{v\in \mathbb {Z}^{d}}$ is i.i.d.,
so long as $\eta (v)$ is not deterministically
equal to zero \cite{AMPR}.
On the other hand, the frog model on the infinite $d$-ary tree can be either transient
a.s.\ or recurrent a.s., depending on the initial configuration.
For example, on the $d$-ary tree when $(\eta (v))_{v}$ is
i.i.d.-$\Poi (\mu )$, the frog model
is recurrent or transient depending on whether $\mu $ is greater or less than a critical
value $\mu _{c}(d)$ \cite{HJJ2,JJ_log}. In \cite{JJ_order}, the authors give a theorem
comparing frog models with different initial conditions on the same graph, which shows that
the frog model on the $d$-ary tree is recurrent if $\eta (v)\pgfsucc \Poi (\mu )$
for all $v$ for some $\mu >\mu _{c}(d)$.
The condition $X\pgfsucc Y$ means that $\mathbf {E}t^{X}\leq \mathbf {E}t^{Y}$ for all
$t\in (0,1)$, which roughly speaking requires $\eta (v)$ to have expectation at least $\mu $
and be less dispersed than the Poisson distribution.

This raises the question of whether the transience and recurrence of the frog model depends
on the entire distribution of each $\eta (v)$ or just on the expectation. This question was posed
in \cite[Open~Question~11]{JJ_order} with the conjecture that the entire distribution
matters. We confirm this:
\begin{theorem}\label{thm:transience}
  Consider the frog model on the the $d$-ary tree with i.i.d.-$\pi $ initial conditions.
  For arbitrarily large $\mu $, there exists a distribution $\pi $ with mean~$\mu $ so that the model
  is a.s.\ transient.
\end{theorem}

Thus, for large enough $\mu $, there exist distributions $\pi $ with mean $\mu $ so that the frog model on the $d$-ary tree is recurrent or so that it is transient.
In \S \ref{sec:proof} we prove Theorem~\ref{thm:transience}.
In \S \ref{sec:more} we describe how the argument can be extended to show transience for some initial distributions with infinite mean and (for $d \geq 14$) even when every site contains at least one particle.

For two processes resembling the frog model, the long-time behavior of the model depends only
on the mean particle distribution. The first of these is \emph{branching random walk}, essentially the frog
model except that particles spawn new particles even when moving to a previously visited site.
It is a classical result of Biggins's that for BRW, recurrence vs.\ transience depends only
on the mean number of particles spawned.
We explain this in more detail in Remark~\ref{prop:mean}.

The other process with contrasting behavior is
\emph{activated random walk}. Particles in this model move with the same dynamics as
the frog model except that active particles have some probability
of falling back asleep at each step. This process is usually considered
with all particles starting awake and moving in continuous time.
Both ARW and the frog model have been of significant interest in physics.
The frog model and the broader family
of $A+B\to 2B$ models that it falls into have been viewed as stochastic combustion models.
They have also been investigated as part of the general study of propagating fronts;
see Section~2.6.1 (i)--(v) in \cite{Panja} for a survey of the physics literature.
The interest in ARW comes from a connection with the phenomenon of self-organized criticality.
For a given sleep rate, ARW has been shown on many graphs to undergo
an absorbing-state phase transition; see the introductions of~\cite{BasuGangulyHoffman18,RS,ST,StaTag}.
In \cite{RSZ}, it is shown that the phase transition for this model on a regular tree with all particles initially active depends only on the initial density of particles, not on the distribution of particles at each site (it is stated for $\mathbb{Z}^{d}$ but the argument also works for unimodular graphs).
On the other hand, by Theorem~\ref{thm:transience} the same model starting with all but one particle sleeping can fixate with arbitrarily high densities.
This suggests that the question of fixation vs.\ activity for the ARW is more delicate than it may seem, at least outside the setup of graphs with polynomial growth.

\section{Proofs}
\label{sec:proof}

We consider the $(d+1)$-regular tree, denoted $\Thom _{d}$, and in the end we translate the result from $\Thom _{d}$ to the $d$-ary tree.
Let $\rho \in \Thom _{d}$ be the designated vertex starting with one active particle, which we will call the root.
Fix a choice of $\mu $, and let
\[ \pi _{N}\defeq \frac{\mu }{N}\, \delta _{N} + \bigl(1-\frac{\mu }{N}\bigr)\, \delta _{0}.\]
From now on, we consider the frog model on $\Thom _{d}$ with i.i.d.-$\pi _{N}$ initial conditions.
The initial configuration can be thought of as a sea of empty sites, with islands on which $N$ particles sleep.
Our goal is to show that for large enough $N$, this frog model is transient.

The typical way of proving transience for the frog model on a tree is to
add extra particles to the model
so that a particle \emph{always} wakes up particles when it moves.
The resulting process is a branching random walk, which can then shown
to be transient when the distribution of sleeping particles is sufficiently small.
See \cite[Proposition~15]{HJJ2} for the basic example of this argument, or see
\cite[Section~3.2]{HJJ1} for a more elaborate one.
The recurrence or transience of the branching random
walk produced by this argument
depends only on the expected number of particles per site (see Remark~\ref{prop:mean}).
Thus, we will need a different
approach, since for large values of $\mu $ this branching random walk will always be recurrent.

Our argument instead uses a different branching process not indexed by time
in the usual way.
Rather than track individual particles and branch at each site they visit,
we branch at each site visited over all time by the batch of particles starting
from the same island.
The advantage is that if many particles starting from the same island visit the same site,
the process branches only once at that site.

\begin{figure}
  \centering
\includegraphics[width=.7\textwidth ]{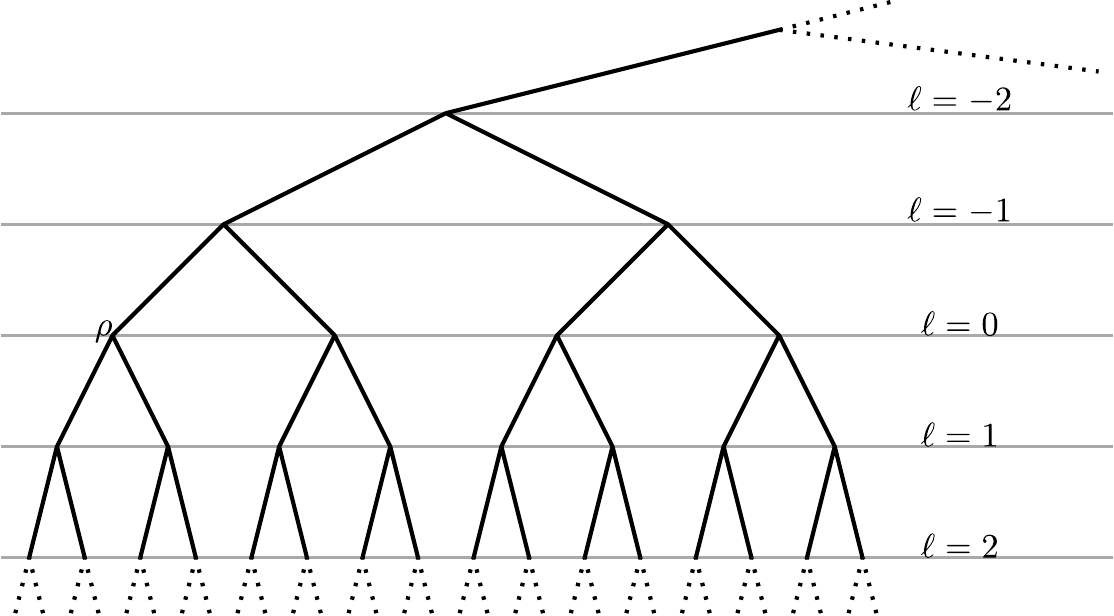}
  \caption{A neighborhood of the root $\rho $ in $\Thom _{2}$, organized by levels.}
  \label{fig:levels}
\end{figure}

We define a function $\ell \colon \Thom _{d}\to \mathbb {Z}$ where $\ell (v)$ represents the ``level'' of
vertex~$v$ as follows.
We set $\ell (\rho )=0$, and then specify that for each $v\in \Thom _{d}$, we have
$\ell (u)=\ell (v)-1$ for one neighbor $u$ of $v$ and $\ell (u)=\ell (v)+1$ for all other neighbors
$u$ (see Figure~\ref{fig:levels}). Note that $\ell (v)$ is not the same as the distance between $v$ and $\rho $.\footnote{This artificial introduction of levels is the same as giving an arbitrary genealogy to the tree, so that the level decreases towards ancestors and increases towards descendants.
In particular, the distance from $\rho $ to $v$ equals $\pm \ell (v)$ if and only if $v$ is a direct descendant or ancestor of $\rho $, respectively.}
Fix $\lambda >0$.
We define the weight function
\begin{align*}
  w_\lambda (A) = \sum _{v\in A} \lambda ^{\ell (v)}
, \quad
A\subseteq \Thom _{d}
.
\end{align*}

Most of our work will be in the following lemma.

\begin{lemma}\label{lem:mtg}
Suppose a random number of particles distributed as $\pi _{N}$ start at $\rho $ and perform independent random walks on $\Thom _{d}$.
Let $A$ be the random (and possibly empty) set of sites visited by these particles.
Then, for $\lambda =\frac{1}{\sqrt{d}}$ and $N$ sufficiently large, $\mathbf {E}w_\lambda (A) < 1$.
\end{lemma}

Choosing $\frac{1}{d}<\lambda <1$ would also work but the above choice simplifies the computations, so we fix $\lambda =\frac{1}{\sqrt{d}}$. This choice also asymptotically minimizes the bound we prove
as $N\to \infty $.
Before we prove this estimate, we show how it implies Theorem~\ref{thm:transience}.

\begin{proof}[Proof of Theorem~\ref{thm:transience} assuming Lemma~\ref{lem:mtg}]
  To prove a.s.\ transience, it is enough to do so assuming $\eta (\rho )$ is also random and distributed as $\pi _{N}$ like the other sites.
Define $\mathcal{B}_{0}=\{\rho \}$.
Define $\mathcal{B}_{n+1}$ inductively as the set of sites outside $\mathcal{B}_{0} \cup \dots \cup \mathcal{B}_{n}$ that are visited at any time by a particle originating at a site in $\mathcal{B}_{n}$.
The set $\cup _{n=0}^{\infty }\mathcal{B}_{n}$ then consists of all sites ever visited by the process.
We will show that this set does not encompass all of $\Thom _{d}$, from which we can conclude that the model is transient.
Although the sets $\mathcal{B}_{n}$ will be infinite when $\eta (\rho ) \ge 1$, their weight is finite, and using Lemma~\ref{lem:mtg} recursively we show that they are in fact small.

  Take $N$ large enough that $\alpha :=\mathbf {E}w_\lambda (A)<1$
  for the random set $A$ defined in Lemma~\ref{lem:mtg}.
  Suppose $v\in \mathcal{B}_{n}$ for some $n$.
The set of sites visited by the particles originating at $v$ is distributed identically to $A$, except that it is shifted from starting at $v$ instead of the root.
Thus, the expected weight of the sites visited by the particles originating at $v$ is $\lambda ^{\ell (v)}\mathbf {E}w_\lambda (A)$.
Hence,
  \begin{align*}
    \mathbf {E}\bigl[ w_\lambda (\mathcal{B}_{n+1})\mid \mathcal{B}_{n}\bigr] &\leq \sum _{v\in \mathcal{B}_{n}} \lambda ^{\ell (v)}\mathbf {E}w_\lambda (A) = \alpha w_\lambda (\mathcal{B}_{n}).
  \end{align*}
Since $w_\lambda (\mathcal{B}_{0})=1$, we get $\mathbf {E}[w_\lambda (\mathcal{B}_{n})]\leq \alpha ^{n}$ and finally
\begin{align*}
  \mathbf {E}\biggl[ w_\lambda \Bigl(\cup _{n=0}^{\infty }\mathcal{B}_{n}\Bigr) \biggr] =
    \mathbf {E}\Biggl[\sum _{n=0}^{\infty } w_\lambda (\mathcal{B}_{n}) \Biggr] \leq \sum _{n=0}^{\infty }\alpha ^{n}<\infty .
\end{align*}
Therefore $w_\lambda \bigl(\cup _{n=0}^{\infty }\mathcal{B}_{n}\bigr)<\infty $ a.s. Since $\Thom _{d}$ has infinite weight, we can conclude that with probability~$1$, not every site is visited.
Moreover, only finitely many sites at each level are ever visited.
We can also conclude that no site is visited infinitely often, since if one were, then
a.s.\ all vertices would be visited. Hence level~0 is visited finitely many times,
which by \cite[Corollary~16]{HJJ1} implies transience for the equivalent frog model on the $d$-ary tree.
\end{proof}

In what follows, each appearance of $C$ denotes a different positive finite constant which depends on $d$ and whose actual value (sometimes easy to find) is irrelevant.

\begin{proof}[Proof of Lemma~\ref{lem:mtg}]
  Our strategy is to sum $\mathbf {P}[v\in A] \lambda ^{\ell (v)}$ over all $v\in \Thom _{d}$.
  Let $\varphi (j,k)$ be the number of vertices in $\Thom _{d}$ at level~$j$ that have distance~$k$
  from the root. For $j\geq 1$,
  \begin{align*}
    \varphi (j, j) &= d^{j},\\
    \varphi (j,j+2i) &= (d-1) d^{j+i-1}\leq d^{j+i}\text{ for $i\geq 1$,}
  \intertext{and for $j\leq 0$,}
    \varphi (j,\abs{j}) &= 1,\\
    \varphi (j,\abs{j} + 2i) &= (d-1)d^{i-1}\leq d^{i}\text{ for $i\geq 1$.}
  \end{align*}
  If $k\neq \abs{j}+2k$ for some $i\geq 0$, then $\varphi (j,k)=0$.

The above combinatorial terms will be controlled using hitting probabilities.
For any $v\in \Thom _{d}$ at distance~$k$ from the root,
\begin{equation}
\label{eq:p.bound}
\mathbf {P}[\text{a random walk starting at } \rho \text{ ever visits } v] = d^{-k}.
\end{equation}
Let $p_{N}(k)$ be the probability that at least one of $N$ independent random walks starting at the root eventually hits a given vertex at distance~$k$ from the root in $\Thom _{d}$.
By a union bound,
\begin{equation}
\label{eq:pnbound}
p_{N}(k) \leq d^{-k}N.
\end{equation}
We have $v\in A$ if $\eta (\rho )=N$ and one of the $N$ random walks hits $v$.
Assuming $N=d^{m}$,
\[ \mathbf {P}[v\in A] = \frac{\mu }{N} \, p_{N}(k) \le \mu d^{-m} \min \{d^{-k}N,1\} = \mu \min \{d^{-k},d^{-m}\} . \]
Now let $\mathcal {S}=\mathbb {Z}\times \{0,1,\ldots \}$.
We split the vertices according to $(j,i)\in \mathcal {S}$, since each $v\in \Thom _{d}$ is at some level~$j$ and some distance $\abs{j}+2i$ from the root.
Our goal is to bound
\begin{align}\label{eq:wlambda.bound}
    \mathbf {E}w_\lambda (A) &= \sum _{v\in \Thom _{d}}\mathbf {P}[v\in A] \lambda ^{\ell (v)}
\le
\sum _{(j,i)\in \mathcal {S}} \mu \lambda ^{j} \varphi (j, \abs{j}+2i) \min \{d^{-|j|-2i} , d^{-m}\}
.
  \end{align}
  We break the set $\mathcal {S}$ into six parts:
  \begin{align*}
    \mathcal {S}_{1}^{+} &= \bigl\{(j,i)\in \mathcal {S}\colon \quad 1\leq j\leq m,\quad 2i \le m - j \bigr\},\\
    \mathcal {S}_{2}^{+} &= \bigl\{(j,i)\in \mathcal {S}\colon \quad 1\leq j\leq m,\quad 2i > m - j \bigr\},\\
    \mathcal {S}_{3}^{+} &= \bigl\{(j,i)\in \mathcal {S}\colon \quad j>m \bigr\},\\
    \mathcal {S}_{1}^{-} &= \bigl\{(-j,i)\in \mathcal {S}\colon \quad 0\leq j\leq m,\quad 2i \le m - j \bigr\},\\
    \mathcal {S}_{2}^{-} &= \bigl\{(-j,i)\in \mathcal {S}\colon \quad 0\leq j\leq m,\quad 2i > m - j \bigr\},\\
    \mathcal {S}_{3}^{-} &= \bigl\{(-j,i)\in \mathcal {S}\colon \quad j>m \bigr\}.
  \end{align*}
The sets $\mathcal {S}_{3}^\pm $ are the easiest to estimate.
Let $e_{j,i} = \lambda ^{j} \varphi (j, \abs{j}+2i) \min \{d^{-|j|-2i} , d^{-m}\}$.
Then
\begin{align}
\nonumber
\sum _{(j,i)\in \mathcal {S}_{3}^{+}} e_{j,i}
\leq
\sum _{j > m} \lambda ^{j} \sum _{i \ge 0} d^{j+i} d^{-j-2i}
=
C \sum _{j>m} \lambda ^{j}
\end{align}
and
\[
\sum _{(-j,i)\in \mathcal {S}_{3}^{-}} e_{-j,i}
\leq
\sum _{j > m} \lambda ^{-j} \sum _{i \ge 0} d^{i} d^{-j-2i}
=
C \sum _{j>m} (\lambda d)^{-j}
.
\]
Note that both vanish as $m\to \infty $.
For $\mathcal {S}_{1}^\pm $, we have
\begin{align*}
\sum _{(j,i)\in \mathcal {S}_{1}^{+}} e_{j,i}
&\leq
\sum _{1 \le j \le m} \lambda ^{j} \sum _{0 \le 2i \le m-j} d^{j+i} d^{-m}
=
d^{-m} \sum _{1 \le j \le m} (\lambda d)^{j} \sum _{i=0}^{\floor[\big ] {\tfrac{m-j}{2}}} d^{i}
\\&
\le
C d^{-m} \sum _{1 \le j \le m} (\lambda d)^{j} \, d^{\frac{m-j}{2}}
=
C d^{-\frac{m}{2}} \sum _{1 \le j \le m} (\lambda \sqrt{d})^{j}
=
C d^{-\frac{m}{2}} m,
\end{align*}
recalling that we set $\lambda =1/\sqrt{d}$. Similarly,
\begin{align*}
\sum _{(-j,i)\in \mathcal {S}_{1}^{-}} e_{-j,i}
&\leq
\sum _{0 \le j \le m} \lambda ^{-j} \sum _{0 \le 2i \le m-j} d^{i} d^{-m}
=
d^{-m} \sum _{0 \le j \le m} \lambda ^{-j} \sum _{i=0}^{\floor[\big ]{\tfrac{m-j}{2}}} d^{i}
\\&
\le
C d^{-m} \sum _{0 \le j \le m} \lambda ^{-j} \, d^{\frac{m-j}{2}}
=
C d^{-\frac{m}{2}} \sum _{0 \le j \le m} (\lambda \sqrt{d})^{-j}
=
C d^{-\frac{m}{2}} (m+1)
.
\end{align*}
Both expressions again vanish as $m\to \infty $.
Finally,
\begin{align*}
\sum _{(j,i)\in \mathcal {S}_{2}^{+}} e_{j,i}
&\leq
\sum _{1 \le j \le m} \lambda ^{j} \sum _{2i > m-j} d^{j+i} d^{-j-2i}
=
\sum _{1 \le j \le m} \lambda ^{j} \sum _{2i > m-j} d^{-i}
\\&
\le
C \sum _{1 \le j \le m} \lambda ^{j} \, d^{-\frac{m-j}{2}}
=
C d^{-\frac{m}{2}} \sum _{1 \le j \le m} (\lambda \sqrt{d})^{j}
=
C d^{-\frac{m}{2}} m
\end{align*}
and
\begin{align*}
\sum _{(-j,i)\in \mathcal {S}_{2}^{-}} e_{-j,i}
&\leq
\sum _{0 \le j \le m} \lambda ^{-j} \sum _{2i > m-j} d^{i} d^{-j-2i}
=
\sum _{0 \le j \le m} \lambda ^{-j}d^{-j} \sum _{2i>m-j} d^{-i}
\\&
\le
C \sum _{0 \le j \le m} \lambda ^{-j}d^{-j} \, d^{-\frac{m-j}{2}}
=
C d^{-\frac{m}{2}} \sum _{0 \le j \le m} (\lambda \sqrt{d})^{-j}
=
C d^{-\frac{m}{2}} (m+1)
,
\end{align*}
both vanishing like the other terms.
Therefore,
\[
\mathbf {E}w_\lambda (A) \le \mu \sum _{(j,i)\in \mathcal {S}} e_{j,i}
\]
can be made less than $1$ by choosing $m$ large and $N=d^{m}$, which proves the lemma.
\end{proof}

\begin{remark}
\label{prop:mean}
It is a classical result that the recurrence or transience of a branching random walk on the integers
depends only on the expected offspring distribution. We explain now how this result extends to symmetric nearest-neighbor branching
random walk on a regular tree.
To precisely define the BRW on the integers, let $Z$ be a point process on the integers.
The initial generation of the BRW is a single particle at $0$.
To obtain generation~$n+1$, each particle in generation~$n$
places new particles with positions given by an independent copy of $Z$ translated by the particle's position.
We assume that $Z$ always contains at least one point so that the process survives a.s.,
and that it contains a point in the negative integers with positive probability.

To give Biggins's criterion for transience, let $m(\lambda )=\mathbf {E}\sum _{x\in Z} e^{-\lambda x}$, where the sum
is over the atoms of $Z$.
The BRW is transient in the positive direction if and only if there
exists $\lambda >0$ so that $m(\lambda )\leq 1$. See \cite[Theorem~3]{Biggins} for a proof,
though the result goes back to \cite{BigginsMtg}. One can also of course
test for transience in the negative direction by flipping $Z$ across $0$,
and it turns out that transience in the negative and positive directions
and recurrence are the only possibilities. Note that $m(\lambda )$ depends only on the
expected number of atoms in $Z$ at each integer.

Now, consider a BRW on a regular tree starting with a single particle at the root,
where each particle reproduces independently
by placing particles relative to itself sampled from some distribution. Assume that this distribution
places particles only at neighbors of the root and
is invariant under tree automorphisms fixing the root.
Projecting each particle
by $\ell $ yields a BRW on the integers whose transience is determined by
the criterion above. Transience of the BRW on the integers implies transience of the BRW on the
tree. Recurrence of the BRW on the integers implies that level~0 of the tree is visited infinitely often. By this, invariance, and independence of
particles, each particle
in generation~1 almost surely has a descendant that visits level~0. For a generation~1 particle
at level~$1$, this implies that one of its descendant visits the root again.
By invariance under tree automorphisms, this is also true of a particle at level~$-1$.
Thus the root is visited infinitely often a.s.
\end{remark}

\section{Generalizations of the argument}
\label{sec:more}

We now sketch two extensions of the previous argument.
First and very simple is to show transience on $\Thom _{d}$ when the particles per site has infinite mean.
Second and more delicate is to show transience on $\Thom _{d}$ with $d \geq 14$ when the number of
particles not only has infinite mean but is a.s.\ at least 1.

Fix some $\mu >0$.
In the proof of Lemma~\ref{lem:mtg} we have in fact shown that $\mathbf {E}_{\pi _{N}} w_\lambda (A)$ can be made arbitrarily small as $N\to \infty $, where we use $\mathbf {E}_{\pi _{N}}$ to denote the expectation assuming the distribution of
particles is $\pi _{N}$.
Choose $N_{n}$ so that $\mathbf {E}_{\pi _{N_{n}}} w_\lambda (A)<2^{-n}$. Let $X_{n}$ be distributed as $\pi _{N_{n}}$,
let $X=\sum _{n=1}^{\infty }X_{n}$, and let $\pi $ be the distribution of $X$. Then
\begin{align}\label{eq:inf.trick}
\mathbf {E}_\pi w_\lambda (A) \le \sum _{n} \mathbf {E}_{\pi _{N_{n}}} w_\lambda (A) < 1,
\end{align}
and the proof of Theorem~\ref{thm:transience} carries on without any change to show transience with particle
counts given by $\pi $, even though $\pi $ has expectation $\sum _{n} \mu = \infty $.

Now suppose $d \ge 14$. We start by showing that the model is transient when the initial distribution
is given by $\pi _{N}$ plus an extra particle at each site. We combine
the approach introduced in \S \ref{sec:proof} with the standard one based on branching random walks.
First take $\zeta (v)$ independent with distribution $\pi _{N}$, and consider the configuration $\eta $ given by $\eta (\rho )=\zeta (\rho )$ and $\eta (v)=\zeta (v)+1$ for $v\ne \rho $ (as before,
altering the distribution at the root does not affect the property of a.s.\ transience).
We assign the guaranteed particle at each site \emph{type~1} and the $\zeta (v)$ particles \emph{type~2}.
Start with $\mathcal{B}_{0}=\{\rho \}$.
Define $\mathcal{B}_{n+1}$ inductively as follows.
Launch the type~2 particles from each site $v\in \mathcal{B}_{n}$, and allow only the type~1 particles to wake.
Take $\mathcal{B}_{n+1}$ as the set of sites outside of $\mathcal{B}_{0} \cup \dots \cup \mathcal{B}_{n}$ that are visited.
As before, the set $\cup _{n=0}^{\infty }\mathcal{B}_{n}$ consists of all sites ever visited in the frog model.
Moreover, $\mathbf {E}w_\lambda (\mathcal{B}_{n}) \le \alpha ^{n}$ where $\alpha :=\mathbf {E}w_\lambda (\mathcal{B}_{1})$.

So again it suffices to show that $\alpha <1$.
Instead of the hitting probability~\eqref{eq:p.bound}, we use the following estimate
based on BRWs:
\begin{lemma}
\label{lem:brw}
  Let $d\geq 6$, and let $v\in \Thom _{d}$ be an arbitrary vertex at distance $k$ from the root.
  Run the frog model with one sleeping frog per site starting with the particle at the root
  active.
  Then the
  probability of ever hitting $v$ is at most
  \begin{align*}
    \frac{d+1}{d+1-\sqrt{8d}}\biggl(\frac{4}{d+1}\biggr)^{k}.
  \end{align*}
\end{lemma}

For $d\geq 14$, we have $4/(d+1)<d^{-1/2}$.
By the above lemma, if a single type~2 particle begins at the root and the frog model
runs with only type~1 particles allowed to wake, then $v$ at level~$k$ is hit with probability
less than $C d^{-\beta k}$
for some $\beta >1/2$ (constants $C$ now depend on $d$ and $\beta $).
Let $\varepsilon =1-\beta <1/2$, and let $\beta '=1-2\varepsilon >0$.
Instead of~\eqref{eq:pnbound}, we have $p_{N}(k) \leq C d^{-\beta k}N$,
and we consider the sum over $(j,i)\in \mathcal {S}$ of
\[
e_{j,i} = C \lambda ^{j} \varphi (j, \abs{j}+2i) \min \{d^{-\beta (|j|+2i)} , d^{-m}\}
\]
which is the same as before with an extra $\beta $.
The estimates in $\mathcal {S}_{1}^\pm $ of course do not change.
The estimate in $\mathcal {S}_{2}^{+}$ and $\mathcal {S}_{3}^{+}$ becomes $e_{j,i} \le C (\lambda d^\varepsilon )^{j} d^{-\beta 'i}$.
The estimate in $\mathcal {S}_{2}^{-}$ and $\mathcal {S}_{3}^{-}$ becomes $e_{-j,i} \le C (\lambda d^\beta )^{-j} d^{-\beta 'i}$.
Since these are still summable over $\mathcal {S}$, the sum over $\mathcal {S}_{3}^\pm $ vanishes for large $m$ as before.
Carrying the same computations as in the proof of Lemma~\ref{lem:mtg} we get as upper bounds
$C d^{-\beta '\frac{m}{2}} m$
for the sum over
$\mathcal {S}_{2}^{+}$
and
$C d^{-\beta '\frac{m}{2}} (m+1)$
for the sum over
$\mathcal {S}_{2}^{-}$.
Both vanish for large $m$, and therefore $\alpha $ can be made arbitrarily small.
This proves that the model is transient with distribution $\pi _{N}$ plus an extra particle
for large enough $N$. Using the same idea as in \eqref{eq:inf.trick}, we can maintain $\alpha <1$
while replacing $\pi _{N}$ with a distribution with infinite expectation, and the proof goes on as before.

\begin{proof}
[Proof of Lemma~\ref{lem:brw}]

Consider the following BRW.
One particle starts at the root $\rho \in \Thom _{d}$. At every step, each particle at $u$
chooses a neighbor $u'$ of $u$ uniformly at random. If $u'$ lies on the path from $\rho $ to $u$ (including
$\rho $ itself), the particle produces one offspring at $u'$; if not, it produces two offspring at $u'$.
Since a particle moving towards the root in the frog model never wakes any particles, this BRW
dominates the frog model with one sleeping particle per site,
in the sense that it can be coupled with it so that every active particle
in the frog model at time~$n$ also exists in the BRW. Thus it suffices to bound
the probability that this BRW hits $v$.

  Assign a particle at distance~$k$ from the root to have weight $e^{-\theta k}$,
  for $\theta $ to be chosen later. Let $W_{n}$ be the total weight of all particles
  in the BRW at time~$n$. Define
  \begin{align*}
    m = \frac{1}{d+1}e^\theta + \frac{2d}{d+1}e^{-\theta }.
  \end{align*}
  We claim that
\begin{align}\label{eq:mtg}
    \mathbf {E}[W_{n+1}\mid W_{n}]\leq mW_{n}.
  \end{align}
  Indeed, consider a particle at distance~$k$ from the root
  at time~$n$, which has weight $e^{-\theta k}$. If $k\geq 1$, then at its next step it produces
  a single particle with weight $e^{-\theta (k-1)}$ with probability~$1/(d+1)$, or it produces
  two particles each with weight $e^{-\theta (k+1)}$ with probability~$d/(d+1)$. Thus the expected
  weight from the offspring is exactly $m$ times the weight of the vertex. If $k=0$, then the particle
  at its next step deterministically produces two offspring with weight $e^{-\theta }$, which collectively
  have weight $2e^{-\theta }<m$. Thus, the expected weight of the offspring of any particle
  is bounded by $m$ times the weight of the particle, which proves~\eqref{eq:mtg}.
  Since $W_{0}=1$, this shows that $\mathbf {E}W_{n} \leq m^{n}$.
  Set $\theta =\log (2d)/2$ to optimize $m$. This gives $m=\sqrt{8d}/(d+1)$. Under our assumption
  $d\geq 6$, we have $m < 1$.

  Now, let $X_{k}$ denote the number of vertices at distance~$k$ from the root that are ever
  visited by the BRW. A particle at distance~$k$ from the root has weight $e^{-\theta k}$
  at that time, and the first time that there can be a particle at distance~$k$ from the root
  is time~$k$. Thus, $e^{-\theta k}X_{k}\leq \sum _{n=k}^{\infty } W_{n}$. Taking expectations,
  \begin{align*}
    \mathbf {E}X_{k} \leq e^{\theta k}\frac{m^{k}}{1-m}
      =\frac{1}{1-m} \biggl(\frac{4d}{d+1}\biggr)^{k}.
  \end{align*}
  Since $\mathbf {E}X_{k}$ is the sum of the probabilities of each the $(d+1)d^{k-1}$ vertices at distance~$k$
  from the root being visited,
  \begin{align*}
    \mathbf {P}[\text{$v$ is visited}] &= \frac{1}{(d+1)d^{k-1}}\mathbf {E}X_{k} \leq
       \frac{1}{1-m}\biggl(\frac{4}{d+1}\biggr)^{k}. \qedhere
  \end{align*}
\end{proof}

\ACKNO{T.J.\ thanks the support and hospitality of the NYU-ECNU Institute of Mathematical Sciences at NYU Shanghai.
L.R.\ received support from grants UBACYT-2017 Mod-I 20020160100147BA and PICT 2015-3154.
T.J.\ received support from NSF grants DMS-1401479 and DMS-1811952 and PSC-CUNY award \#61540-00~49.
}







\end{document}